\pdfoutput=1
\documentclass[10pt,letterpaper,leqno,twoside]{article}
\usepackage[colorlinks=true, pdfstartview=FitV, linkcolor=blue, citecolor=blue, urlcolor=blue]{hyperref}
\usepackage[latin1]{inputenc}
\usepackage[OT2,OT1]{fontenc}
\DeclareFontFamily{OT1}{rsfs}{}
\DeclareFontShape{OT1}{rsfs}{n}{it}{<-> rsfs10}{}
\DeclareMathAlphabet{\mathscr}{OT1}{rsfs}{n}{it}
\usepackage[english]{babel}
\usepackage{amsmath}
\usepackage{amsfonts}
\usepackage{amssymb}
\usepackage{amsthm}
\usepackage{graphicx}
\usepackage[protrusion=true,expansion=true]{microtype}
\usepackage[left=2.3cm,right=2.3cm,top=2.8cm,bottom=2.8cm]{geometry}
\author{Dimitri Dias}
\title{The angular distribution of integral ideal numbers with a fixed norm in quadratic extensions}
\date{}

\newtheorem{theorem}{Theorem}
\newtheorem{lemma}{Lemma}
\newtheorem{definition}{Definition}
\newcommand{\e}{\mathrm{e}}
\newcommand{\dd}{\, \mathrm{d}}

\newcounter{numrem}
\newenvironment{remark}[1][Remark \arabic{numrem} :]{\refstepcounter{numrem} \begin{trivlist}
\item[\hskip \labelsep {\bfseries #1}]}{\end{trivlist}}

\begin{document}
\maketitle

\abstract{In \cite{Hall}, Erd\H{o}s and Hall studied the angular distribution of Gaussian integers with a fixed norm. We generalize their result to the angular distribution of integral ideal numbers with a fixed norm in any quadratic extension.}

\section{Introduction}

In \cite{Hall}, Erd\H{o}s and Hall studied the angular distribution of Gaussian integers with a fixed norm, which, geometrically, is the same as studying the distribution of points with integral coordinates on a circle. Even though the arguments of Gaussian integers with a fixed norm are not uniformly distributed in $[0,2\pi)$, Erd\H{o}s and Hall showed that their discrepancy is fairly small.
\\

One can extend their result to other quadratic extensions. Since $\mathbb{Z}[i]$ is a principal ideal domain, the result mentioned previously can be seen as a study of the arguments of the generators of the integral ideals of a given norm over $\mathbb{Z}[i]$. In a general quadratic extension, the ring of integers might not be a principal ideal domain. However, in \cite{Heck1,Heck2}, Hecke introduced the concept of ideal numbers, based on an idea from Kummer's work on cyclotomic fields. Hecke's idea consists in representing ideals by specific algebraic numbers. More precisely, to any integral ideal we can associate algebraic numbers which divide any element of the ideal. These are called integral ideal numbers.
\\

In the case of a principal ideal domain, the possible generators of an ideal and the associated ideal numbers coincide. Therefore, the result of Erd\H{o}s and Hall is in fact a study of the argument of integral ideal numbers of a given norm in $\mathbb{Q}(i)$. In this paper, we generalize, using their method, their result to any quadratic extension.

\subsection*{\small{Acknowledgements}}
I would like to thank my supervisor, Professor Andrew Granville, for discussions that greatly helped me with my work. I would also like to thank my friends Oleksiy Klurman, Crystel Bujold, Kevin Henriot and Marzieh Mehdizadeh for their helpful comments.

\section{Generalities}

\subsection{Ideal numbers}

We first recall the basic definitions and properties of ideal numbers. For more details, see, for example, \cite{Heck1,Heck2}.
\\

Let $\mathbb{K}$ be an arbitrary number field with class number $h$, let $\mathcal{O}_{\mathbb{K}}$ be the ring of algebraic integers of $K$ and let $\mathcal{U}$ be the set of units of $\mathcal{O}_{\mathbb{K}}$. In what follows, we consider ideals (integral of fractional) over $\mathcal{O}_{\mathbb{K}}$. It is well known that the classes of ideals over $\mathcal{O}_{\mathbb{K}}$ form a finite abelian group of order $h$, called the class group.
\\
\\
If $h > 1$, the class group has a basis, say the classes $B_1, \dots, B_k$, with order $c_1, \dots, c_k$. In each class $B_i$, we choose one ideal $\mathfrak{b}_i$. By definition of a basis, every ideal is equivalent to a unique product $\mathfrak{b}_1^{m_1} \dots \mathfrak{b}_k^{m_k}$ with $0 \leq m_i < c_i$. Therefore, for any ideal $\mathfrak{a}$, there exists $c \in \mathbb{K}$, such that
\begin{equation*}
\mathfrak{a} = c \mathfrak{b}_1^{m_1} \dots \mathfrak{b}_k^{m_k} \, .
\end{equation*}
Now, $\mathfrak{b}_i^{c_i} = [\beta_i]$ for each $1 \leq i \leq k$, with $\beta_i \in \mathbb{\mathcal{O}_{\mathbb{K}}}$. Define $\omega_i=\sqrt[c_i]{\beta_i}$ where we fix some choice of $c_i$-th root of $\beta_i$. Define
\begin{equation*}
\alpha = c \omega_1^{m_1} \dots \omega_k^{m_k} \, .
\end{equation*}
$\alpha$ is called an ideal number associated to the ideal $\mathfrak{a}$.
\\
\\
If $h=1$, any ideal is principal. If $\mathfrak{a}=(a)$ is an ideal, the set of ideal numbers associated to $\mathfrak{a}$ is defined as the set of $\vert \mathcal{U} \vert$ elements of the form $u a$, with $u$ a unit of $\mathcal{O}_{\mathbb{K}}$.
\\
\\
One can easily extend the multiplication from the set of ideals to the set of ideal numbers. Notice that the map from the set of ideal numbers to the set of ideals given by
\begin{equation*}
c \omega_1^{m_1} \dots \omega_k^{m_k} \longmapsto \mathfrak{a} = c \mathfrak{b}_1^{m_1} \dots \mathfrak{b}_k^{m_k}
\end{equation*}
is a surjective homomorphism with kernel $\mathcal{U}$, and is therefore $\vert \mathcal{U} \vert$-to-$1$.
\\
\\
Ideal numbers split into $h$ different classes, corresponding to the ideal classes. One can also easily prove that the set of ideal numbers associated to $\overline{\mathfrak{a}}$ is the set of complex conjugates of the ideal numbers associated to $\mathfrak{a}$.
\\
\\
An ideal number $\alpha$ is said to be integral if $\alpha$ is an algebraic integer, i.e. if and only if the corresponding ideal is an integral ideal. The norm on an integral ideal number $\alpha$ is defined as the norm of the corresponding ideal. An integral ideal number $\alpha$ is said to be a prime ideal number if the corresponding ideal is a prime ideal.
\\
\\
Notice that the set of integral ideal numbers inherits a property of factorization in prime ideal numbers from the corresponding property on integral ideals. This factorization is not unique, since each prime ideal number can be chosen up to a unit. However, if we impose the prime ideal numbers appearing in the prime factorization to have argument in $[0,\frac{2\pi}{\vert \mathcal{U} \vert})$, the factorization becomes unique.

\subsection{Argument of an integral ideal number in a quadratic extension}

Let $d>0$ such that $-d$ is a fundamental discriminant. Let $\mathcal{O}_d$ be the ring of integers of $\mathbb{K}=\mathbb{Q}(i \sqrt{d})$, $h$ be the class number and $\mathcal{U}$ be the set of units. We recall that
\begin{equation}
\mathcal{U}= \lbrace u^k \rbrace \text{ with } \begin{cases}
u=i \text{ and } 0 \leq k \leq 3 \text{ if } d=1, \\
u=\e^{i\frac{\pi}{3}} \text{ and } 0 \leq k \leq 5 \text{ if } d=3, \\ 
u=-1 \text{ and } 0 \leq k \leq 1 \text{ else }.
\end{cases}
\end{equation}
Let $n$ be an integer. Write
\begin{equation*}
n= \prod_{\left( \frac{-d}{p}\right)=1} p^{\alpha_p} \prod_{\left( \frac{-d}{q}\right)=0} q^{\beta_q} \prod_{\left( \frac{-d}{r}\right)=-1} r^{\gamma_r}
\end{equation*}
where $\left( \frac{-d}{.}\right)$ denotes the Kronecker symbol. It is classical that
\begin{align*}
&[p]= \mathcal{P} \overline{\mathcal{P}} \text{ with } \mathcal{P} \text{ a prime ideal such that } \mathcal{P} \ne \overline{\mathcal{P}} \text{ if } \left( \frac{-d}{p}\right)=1 \, , \\
&[q]= \mathcal{Q}^2 \text{ with } \mathcal{Q} \text{ a prime ideal } (\text{we have } \mathcal{Q} = \overline{\mathcal{Q}}) \text{ if } \left( \frac{-d}{q}\right)=0 \, , \\
&[r] \text{ is a prime ideal if } \left( \frac{-d}{r}\right)=-1 \, .
\end{align*}
We can easily deduce that, for $n$ to be the norm of an ideal over $\mathcal{O}_d$, we need to have $\gamma_r$ even. We write $\gamma_r=2 \gamma_r'$ .
\\
\\
The prime ideal numbers corresponding to $\overline{\mathcal{P}}$ are the conjugates of those corresponding to $\mathcal{P}$. Therefore, amongst the prime ideal numbers corresponding to $\mathcal{P}$ and $\overline{\mathcal{P}}$, there is exactly one with argument in $[0,\frac{\pi}{\vert \mathcal{U} \vert})$. We denote it by $\pi_p$ and denote by $\phi_p$ its argument.
\\
\\
Write $\eta_q$ for the prime ideal number corresponding to $\mathcal{Q}$ with argument in $[0,\frac{\pi}{\vert \mathcal{U} \vert})$. In fact, since $\mathcal{Q}=\overline{\mathcal{Q}}$, $\eta_q$ has argument $0$. Then, every integral ideal number of norm $n$ is of the form
\begin{equation*}
u^k \prod_{\left( \frac{-d}{p}\right)=1} \pi_p^{\delta_p} \overline{\pi_p}^{\alpha_p - \delta_p} \prod_{\left( \frac{-d}{q}\right)=0} \eta_q^{\beta_q} \prod_{\left( \frac{-d}{r}\right)=-1} r^{\gamma_r'}
\end{equation*}
where $0 \leq k \leq \vert \mathcal{U} \vert -1$ and any element of this form is an integral ideal number of norm $n$.

\begin{definition}
Let $r(n)$ be the number of integral ideal numbers of norm $n$. We denote $\mathbb{G}$ the set of integers $n$ with $r(n)>0$.
\end{definition}

\begin{remark} \label{rmk1}
From the previous observations, $r(n)>0$ if and only if $n$ is of the form 
\begin{equation*}
n=\prod_{\left( \frac{-d}{p}\right)=1} p^{\alpha_p} \prod_{\left( \frac{-d}{q}\right)=0} q^{\beta_q} \prod_{\left( \frac{-d}{r}\right)=-1} r^{2\gamma_r'}
\end{equation*}
and, in this case,
\begin{equation*}
r(n)=\vert \mathcal{U} \vert \sum_{\delta_p} 1 = \vert \mathcal{U} \vert \prod_{\left( \frac{-d}{p}\right)=1} (\alpha_p + 1)
\end{equation*}
where the sum is over the possible choices of 0 $\leq \delta_p \leq \alpha_p$.
\end{remark}

\begin{lemma}
Let $n \in \mathbb{G}$, $n=\prod_{\left( \frac{-d}{p}\right)=1} p^{\alpha_p} \prod_{\left( \frac{-d}{q}\right)=0} q^{\beta_q} \prod_{\left( \frac{-d}{r}\right)=-1} r^{2\gamma_r'}$. Let $\alpha$ be an integral ideal number of norm $n$, i.e.
\begin{equation*}
\alpha=u^k \prod_{\left( \frac{-d}{p}\right)=1} \pi_p^{\delta_p} \overline{\pi_p}^{\alpha_p - \delta_p} \prod_{\left( \frac{-d}{q}\right)=0} \eta_q^{\beta_q} \prod_{\left( \frac{-d}{r}\right)=-1} r^{\gamma_r'}
\end{equation*}
for some $0 \leq \delta_p \leq \alpha_p$ and for some $0 \leq k \leq \vert \mathcal{U} \vert -1$.
Then,
\begin{equation*}
\textrm{arg}(\alpha)=\sum_{\left( \frac{-d}{p}\right)=1} (2\delta_p - \alpha_p)\phi_p + k \textrm{arg}(u) \bmod{2\pi} \, .
\end{equation*}
Reciprocally, any angle of this form is the angle of some integral ideal number of norm $n$.
\end{lemma}

In Theorem 4 of \cite{Kubi}, Kubilyus states the Prime Ideal Number Theorem in sectors. This result gives information about the distribution of prime ideal numbers in a given class in angular sectors. We give here a version of this theorem with a weaker error term, which will be enough for our proof.
\begin{theorem}[Kubilyus] \label{kub}
Let $x>3$ and $0 \leq \theta_1 \leq \theta_2 \leq 2\pi$. Let $\Pi (\theta_1, \theta_2 ,x)$ be the number of prime ideal numbers of $\mathbb{Q}(i\sqrt{d})$ of norm less than $x$, in a given class, with argument between $\theta_1$ and $\theta_2$. Then,
\begin{equation*}
\Pi (\theta_1, \theta_2 ,x) = \frac{(\theta_2 - \theta_1)\vert \mathcal{U} \vert}{2\pi h} \int_2^x \frac{\dd u}{\log u} + O\left\lbrace x \exp \left( -c \sqrt{ \log x} \right) \right\rbrace
\end{equation*}
for some positive constant $c$.
\end{theorem}

We need two additional facts about the set $\mathbb{G}$. The first one, due to Bernays \cite{Bern}, is an asymptotic for the cardinality of the elements in $\mathbb{G}$ less than $x$, while the second one gives us information about the number of prime divisors of the integers in $\mathbb{G}$.
\begin{lemma} \label{Bern}
There exists a constant $\kappa_d$, depending only on $d$, such that
\begin{equation*}
\vert \mathbb{G} \cap [0,x] \vert = \kappa_d \frac{x}{\sqrt{\log x}} + O \left( \frac{x}{\log^{3/4} x} \right) \, .
\end{equation*}
\end{lemma}

\begin{lemma} \label{omeg}
Let $\omega(n)$ be the number of prime divisors of $n$. Then, all but $o\left( \frac{x}{\sqrt{\log x}} \right)$ integers less than $x$ in $\mathbb{G}$ satisfy $\omega(n) > \left( \frac{1}{2} - \varepsilon \right) \log \log x$.
\end{lemma}

\begin{proof}
On can show, using the asymptotic of Lemma \ref{Bern}, that
\begin{align*}
\sum_{\substack{n \in \mathbb{G} \\ n \leq x}} \omega(n) &= \kappa_d \frac{x}{\sqrt{\log x}} \left( \frac{1}{2} \log \log x + O(1) \right) \, ,\\
\sum_{\substack{n \in \mathbb{G} \\ n \leq x}} \omega(n)^ 2 &= \kappa_d \frac{x}{\sqrt{\log x}} \left( \frac{1}{4} (\log \log x)^2 + O(\log \log x) \right) \, .
\end{align*}
This implies that
\begin{equation*}
\sum_{\substack{n \in \mathbb{G} \\ n \leq x}} (\omega(n) - \frac{1}{2} \log \log x) ^2 = O \left( \frac{x \log \log x}{\sqrt{\log x}} \right) \, .
\end{equation*}
This suffices to prove the result.
\end{proof}

\section{The distribution of ideal numbers with a given norm}

\begin{definition}
Let $\Phi_n$ be the set of arguments of integral ideal numbers of norm $n$.
The discrepancy $\Delta(n)$ of the set $\Phi_n$ is
\begin{equation*}
\Delta(n) = \max \left\lbrace \vert \textrm{card*} \lbrace \phi \in \Phi_n, \phi \in [\theta_1,\theta_2] \bmod{2\pi} \rbrace - \frac{(\theta_2 - \theta_1)}{2\pi} r(n) \vert, 0 \leq \theta_1 < \theta_2 \leq 2\pi \right\rbrace
\end{equation*}
where the asterisk denotes that if the angle is $\alpha$ or $\beta$, it counts for $\frac{1}{2}$.
\end{definition}
\noindent Using the method of Erd\H{o}s and Hall, we will prove the following upper bound for the discrepancy:
\begin{theorem} \label{discr}
Let $\varepsilon > 0$. Then, for all the integers less than $x$ in $\mathbb{G}$, with at most $o\left( \vert \mathbb{G} \cap [0,x] \vert \right)$ exceptions, we have
\begin{equation*}
\Delta(n) \leq \frac{r(n)}{(\log x)^{\frac{1}{2} \log \left( \frac{\pi}{2} \right) - \varepsilon}} \, .
\end{equation*} 
\end{theorem}

\begin{remark}
The trivial upper bound on $\Delta(n)$ is $r(n)$. Theorem \ref{discr} improves it by a power of $\log$. This is still a substantially better bound, since $r(n)$ is itself a power of $\log$. More precisely, for any $n \in \mathbb{G}$, $r(n) \leq d(n)$, the divisor function. Let $\Omega(n)$ be the number of prime divisors of $n$ counted with multiplicity. Similarly to Lemma \ref{omeg}, one can prove that $\Omega(n) < (\frac{1}{2} + \varepsilon)\log \log x$ for almost all $n \in \mathbb{G} \cap [0,x]$. This implies that $r(n) \leq d(n) \leq 2^{\Omega(n)}  \leq (\log x)^{(\frac{1}{2}+\varepsilon)\log 2}$ for almost all the integers $n$ in $\mathbb{G}$ less than $x$.
\end{remark}
To obtain the desired bound for $\Delta(n)$, for a fixed positive real number $y$, we study the average of $\frac{\Delta(n)}{r(n)} y^{\omega(n)}$ over the set $\mathbb{G}$ and show that it is fairly small. We need to use the classical upper bound for the discrepancy, uniform in $T$, due to Erd\H{o}s and Tur\'{a}n \cite{ErTur1,ErTur2},
\begin{equation*}
\Delta(n) \ll \frac{r(n)}{T} + \sum_{t \leq T} \frac{\vert Z_t(n) \vert}{t} \quad \text{ where } \quad	Z_t(n) = \sum_{\phi \in \Phi_n} \e^{i t \phi}
\end{equation*}

\begin{remark} \label{rem2}
\begin{equation*}
Z_t(n) = \sum_{\phi \in \Phi_n} \e^{i t \phi} = \sum_{k=0}^{\vert \mathcal{U} \vert -1} u^{tk} \sum_{\delta_p} \e^{i t \sum (2\delta_p - \alpha_p)\phi_p}
\end{equation*}
where the sum is over the possible choices of 0 $\leq \delta_p \leq \alpha_p$. Therefore,
\begin{equation*}
Z_t(n) = \begin{cases}
\vert \mathcal{U} \vert \sum_{\delta_p} \e^{i t \sum (2\delta_p - \alpha_p)\phi_p} \text{ if } t \equiv 0 \bmod{\vert \mathcal{U} \vert}, \\
0 \text{ if } t \not \equiv 0 \bmod{\vert \mathcal{U} \vert}.
\end{cases}
\end{equation*}
Note that $\frac{\vert Z_t(n) \vert}{\vert \mathcal{U} \vert}$ is multiplicative. Since $\frac{r(n)}{\vert \mathcal{U} \vert}$ is also multiplicative, $\frac{\vert Z_t(n) \vert}{r(n)}$ is a multiplicative function on $\mathbb{G}$.
\end{remark}
Using the previous upper bound,
\begin{equation*}
\sum_{\substack{n \in \mathbb{G} \\ n \leq x}} \frac{\Delta(n)}{r(n)} y^{\omega(n)} \ll \frac{1}{T} \sum_{n \leq x} y^{\omega(n)} + \sum_{\substack{t \leq T \\ t \equiv 0 \bmod{\vert \mathcal{U} \vert}}} \frac{1}{t} \sum_{\substack{n \in \mathbb{G} \\ n \leq x}} \frac{\vert Z_t(n) \vert}{r(n)} y^{\omega(n)} \, .
\end{equation*}
\\
Both sums will be estimated using the following inequality, proven by K{\'a}tai \cite{Kat}:
\begin{lemma} \label{lem2}
Let $f(n)$ be a non-negative multiplicative function, with $f(p^k) \leq Ck$ for every prime power $p^k$. Then,
\begin{equation*}
\sum_{n \leq x} f(n) \leq c \frac{x}{\log x} \exp \left( \sum_{p \leq x} \frac{f(p)}{p} \right)
\end{equation*}
where the constant $c$ only depends on the constant $C$.
\end{lemma}

This inequality can be easily applied to the two sums above. For $p^k \in \mathbb{G}$, from Remark \ref{rmk1} and Remark \ref{rem2}, we have, for $t \equiv 0 \bmod{\vert \mathcal{U} \vert}$,
\begin{equation*}
\frac{\vert Z_t(p^k) \vert}{r(p^k)} y^{\omega(p^k)} \leq y
\quad \text{ and } \quad
\frac{\vert Z_t(p) \vert}{r(p)} y^{\omega(p)} = \begin{cases} y \vert \cos (t \phi_p) \vert \text{ if } \left( -\frac{d}{p} \right)=1 \, , \\
y \text{ if }  \left( -\frac{d}{p} \right)= 0 \, , \\
0 \text{ if }  \left( -\frac{d}{p} \right)= 1 \, .
\end{cases}
\end{equation*}
Therefore, applying Lemma \ref{lem2} to the first and second sum,
\begin{align*}
\sum_{\substack{n \in \mathbb{G} \\ n \leq x}} \frac{\Delta(n)}{r(n)} y^{\omega(n)} &\ll \frac{1}{T} \frac{x}{\log x} \exp \left( y \sum_{\substack{p \leq x \\ \left( -\frac{d}{p} \right)=1}} \frac{1}{p} \right) + \frac{x}{\log x} \sum_{\substack{t \leq T \\ t \equiv 0 \bmod{\vert \mathcal{U} \vert}}} \frac{1}{t} \exp \left( y \sum_{\substack{p \leq x \\ \left( -\frac{d}{p} \right)=1}} \frac{\vert \cos (t \phi_p) \vert}{p} \right) \\
&\ll \frac{1}{T} x (\log x)^{\frac{y}{2}-1} + \frac{x}{\log x} \sum_{\substack{t \leq T \\ t \equiv 0 \bmod{\vert \mathcal{U} \vert}}} \frac{1}{t} \exp \left( y \sum_{\substack{p \leq x \\ \left( -\frac{d}{p} \right)=1}} \frac{\vert \cos (t \phi_p) \vert}{p} \right) \, .
\end{align*}
We are now left with estimating
\begin{equation*}
\sum_{\substack{p \leq x \\ \left( -\frac{d}{p} \right)=1}} \frac{\vert \cos (t \phi_p) \vert}{p} \, .
\end{equation*}
For any $p$, $\theta_p \in (0,\frac{\pi}{\vert \mathcal{U} \vert})$. We split this interval in intervals $E_j$ of length $\frac{\pi}{t}$ such that, on each of these, $\cos (t \theta)$ has a constant sign. More precisely, let
\begin{equation*}
E_j= \left[ \frac{(2j+1)\pi}{2t},\frac{(2j+3)\pi}{2t} \right)  \text{ with } 0 \leq j \leq \frac{t}{\vert \mathcal{U} \vert}-2 \, .
\end{equation*}
Notice that $j$ is indeed an integer since $t \equiv 0 \bmod{\vert \mathcal{U} \vert}$. We have, for any real number $k$,
\begin{equation*}
\sideset{}{^{*}} \sum_{p \leq k} \vert \cos (t \phi_p) \vert = \sideset{}{^{*}} \sum_{\substack{p \leq k \\ \phi_p \in (0,\frac{\pi}{2t})}} \cos (t \phi_p) + (-1)^{\frac{t}{\vert \mathcal{U} \vert} -1} \sideset{}{^{*}} \sum_{\substack{p \leq k \\ \phi_p \in [\frac{\pi}{\vert \mathcal{U} \vert}-\frac{\pi}{2t},\frac{\pi}{\vert \mathcal{U} \vert})}} \cos (t \phi_p) + \sum_{j=0}^{\frac{t}{\vert \mathcal{U} \vert}-2} (-1)^{j+1} \sideset{}{^{*}} \sum_{\substack{p \leq k \\ \phi_p \in E_j}} \cos (t \phi_p)
\end{equation*}
where the asterisk superscript denotes that the sum is restricted to the primes satisfying $\left( -\frac{d}{p} \right)=1$. From Theorem \ref{kub} (where the equality is multiplied by $h$, since we do not consider a specific class), we get that
\begin{align*}
\sideset{}{^{*}} \sum_{\substack{p \leq k \\ \phi_p \in E_j}} \cos (t \phi_p) &= \sideset{}{^{*}} \sum_{\substack{p \leq k \\ \phi_p \in E_j}} \int_{\phi_p}^{\frac{(2j+3)\pi}{2t}} t \sin( t \theta) \dd \theta = \int_{E_j} \sideset{}{^{*}} \sum_{\substack{p \leq k \\ \frac{(2j+1)\pi}{2t} < \phi_p \leq \theta}} t \sin( t \theta) \dd \theta \\
&= \frac{\vert \mathcal{U} \vert}{2\pi} \int_2^k \frac{\dd v}{\log v} \int_{E_j} \left( \theta - \frac{(2j+1)\pi}{2t} \right) t \sin (t \theta) \dd \theta + O\left( k \exp(-c \sqrt{\log k}) \right) \\
&= (-1)^{j+1} \frac{\vert \mathcal{U} \vert}{\pi t} \int_2^k \frac{\dd v}{\log v} +  O\left( k \exp(-c \sqrt{\log k}) \right) \, .
\end{align*}
Estimating the two other sums with a similar method, we can deduce that, for any $2 \leq \omega \leq x$,
\begin{align*}
\sideset{}{^{*}} \sum_{p \leq x} \frac{\vert \cos (t \phi_p) \vert}{p} &\leq \sideset{}{^{*}} \sum_{p \leq w} \frac{1}{p} + \sideset{}{^{*}} \sum_{w < p \leq x}  \frac{\vert \cos (t \phi_p) \vert}{p} \\
&\leq \frac{1}{2} \log \log w + O(1) + \frac{1}{\pi} \log \left( \frac{\log x}{\log w} \right) + O\left( t \exp(-c \sqrt{\log w}) \right)
\end{align*}
using Abel's formula for the second sum. Now, let $\log w = (c^{-1} \log t)^2$ (so that the two error terms are equal). Then, uniformly in $t$,
\begin{equation*}
\sideset{}{^{*}} \sum_{p \leq x} \frac{\vert \cos (t \phi_p) \vert}{p} \leq \frac{1}{\pi} \log \log x + \left( 1 - \frac{2}{\pi} \right) \log \log t + O(1)  \, .
\end{equation*}
Hence,
\begin{align*}
\sum_{\substack{n \in \mathbb{G} \\ n \leq x}} \frac{\Delta(n)}{r(n)} y^{\omega(n)} &\ll \frac{x(\log x)^{\frac{y}{2}-1}}{T} + x(\log x)^{\frac{y}{\pi}-1}(\log T)^{y(1-\frac{2}{\pi}) +1} \\
&\ll \frac{x}{\sqrt{\log x}} (\log \log x)^{\frac{\pi}{2}}
\end{align*}
taking $T= \log x$ and $y=\frac{\pi}{2}$. Now, from Lemma \ref{omeg},
\begin{equation*}
y^{\omega(n)} > y^{\left( \frac{1}{2} - \varepsilon \right) \log \log x} = (\log x)^{(\frac{1}{2}-\varepsilon) \log y}
\end{equation*}
for all but $o\left( \frac{x}{\sqrt{\log x}} \right)$ integers in $\mathbb{G}$. This gives us the desired result.
\section{Angle of a integral representation by a quadratic form}

In what follows, we call a representation of $n$ by $f$ a solution over the integers of the equation $f(x,y)=n$.
\\
\\
It is classical that there is an isomorphism between the classes of primitive integral binary quadratic forms of discriminant $-d$ and the classes of ideals in $\mathcal{O}_d$, which associates the ideal $[a,\frac{-b+i\sqrt{d}}{2}]$ to the form $ax^2 + bxy + cy^2$.
\\

In this section, we will use indifferently the same notation for a class of forms or a class of ideals. More precisely, if $f$ is a reduced primitive binary quadratic form of discriminant $-d$, we write $\mathcal{C}_f$ for the class of forms equivalent to $f$ and for the corresponding class of ideals over $\mathcal{O}_d$. We also denote by $\mathcal{C}_f^{-1}$ its inverse in the ideal class group (and $\mathcal{C}_f^{-1}=\mathcal{C}_{f^{-1}}$, where $f^{-1}$ is called the opposite form).

\begin{definition}
Let $f(x,y)=ax^2 + bxy + cy^2$ be a reduced primitive binary quadratic form of discriminant $-d$ and let $n$ be an integer represented by $f$. Let
\begin{equation*}
f(x_0,y_0)=\frac{N(\alpha_{x_0,y_0})}{a}=n
\end{equation*}
where $N$ is the norm in $\mathcal{O}_d$, $\tau_f=\frac{-b+i\sqrt{d}}{2a}$ and $\alpha_{x_0,y_0} = a(x_0 - \overline{\tau_f} y_0)$, be a representation of $n$ by $f$. We define the angle of this representation of $n$ by $f$ by
\begin{equation*}
\text{arg}(f(x_0,y_0)=n) = \textrm{arg} (\alpha_{x_0,y_0}) \, .
\end{equation*}
\end{definition}

\begin{remark}
Geometrically, the equation $f(x,y)=n$ defines an ellipse. The map $(x,y) \mapsto ( ax+\frac{b}{2}y,\frac{\sqrt{d}}{2})$ maps the ellipse to a circle of radius $\sqrt{n}$, by completing the squares, from which we can parametrize the ellipse. The angle of the representation $f(x_0,y_0)=n$ is the value of this parameter.
\end{remark}

For a reduced primitive binary quadratic form of discriminant $-d$, $f(x,y)=ax^2 + bxy + cy^2$ and an integer $n$ represented by $f$, there is a $\vert \mathcal{U} \vert$-to-one correspondence between the representations of $n$ by $f$ and the ideals of norm $n$ in $\mathcal{C}_f$. More precisely (see, for example, \cite{Cox} for more details), let $f(x_0,y_0)=n$ be a representation of $n$ by $f$. For any unit $u^k$ of $\mathcal{O}_d$, there exists a unique couple $(x'_0,y'_0)$ such that $u^k \alpha_{x_0,y_0}=\alpha_{x'_0,y'_0}$. Then,
\begin{equation*}
f(x'_0,y'_0)=\frac{N(\alpha_{x'_0,y'_0})}{a}=\frac{N(\alpha_{x_0,y_0})}{a}=n
\end{equation*}
is another representation of $n$ by $f$. We call such a representation a representation equivalent to $f(x_0,y_0)=n$. We can group together to these $\vert \mathcal{U} \vert$ equivalent representations and we associate to this set of $\vert \mathcal{U} \vert$ representations the ideal $\alpha_{x_0,y_0}[1,\tau_f]$.

\begin{remark}
The $\vert \mathcal{U} \vert$ different choices for $\alpha_{x_0,y_0}$ differ by units. However, the associated ideal is unique.
\end{remark}

\begin{lemma} \label{transl}
Let $\mathcal{I}$ be the ideal associated to a set of $\vert \mathcal{U} \vert$ representations equivalent to $f(x_0,y_0)=n$. Let $\mathfrak{b}_f=a[1,\overline{\tau_f}]$ be the ideal of the least norm in $\mathcal{C}_{f^{-1}}$. Then,
\begin{equation*}
\left\lbrace \substack{ \text{{\small arguments of ideal}} \\ \text{{\small numbers associated to $\mathcal{I}$}}} \right\rbrace
= 
\left\lbrace \substack{ \text{{\small angles of the $\vert \mathcal{U} \vert$ representations}} \\ \text{{\small equivalent to $f(x_0,y_0)=n$}}} \right\rbrace - \textrm{arg}(\omega_{\mathfrak{b}_f}) \bmod{2\pi}
\end{equation*}
where $\omega_{\mathfrak{b}_f}$ is any ideal number associated to $\mathfrak{b}_f$.
\end{lemma}

\begin{proof}
We have $\mathcal{I}=\alpha_{x_0,y_0}[1,\tau_f]$. Then,
\begin{align*}
\alpha_{x_0,y_0}[1,\tau_f] \mathfrak{b}_f &= \alpha_{x_0,y_0}[1,\tau_f] a[1,\overline{\tau_f}] \\
&= [a \alpha_{x_0,y_0},b \alpha_{x_0,y_0},c \alpha_{x_0,y_0},a \alpha_{x_0,y_0} \tau_f] \\
&= [\alpha_{x_0,y_0}]
\end{align*}
since the norm of the ideal $[\alpha_{x_0,y_0}]$ is $an$, the same as $\alpha_{x_0,y_0}[1,\tau_f] \mathfrak{b}_f$. Therefore, the set of ideal numbers of $\mathcal{I} \mathfrak{b}_f$ is the set of elements of the form $u^k \alpha_{x_0,y_0}$. This proves the result.
\end{proof}

\noindent Geometrically, representations of $n$ by $f$ are points with integral coordinates on the ellipse $f(x,y)=n$. From Lemma \ref{transl}, we immediately deduce that the angular distribution of such points is the same as the angular distribution of ideal numbers of norm $n$ in the class $\mathcal{C}_f$.
\\
\\
Suppose that $h>1$ (since, if there is just one class, there is nothing more to prove than Theorem \ref{discr}).
\begin{definition}
For $f$ a reduced primitive binary quadratic form of discriminant $-d$ and $n \in \mathbb{G}$, $n$ represented by $f$, let $\Phi_n^f$ be the set of arguments of ideal numbers of norm $n$ in $\mathcal{C}_f$ and $r_f(n)$ the number of ideal numbers of norm $n$ in $\mathcal{C}_f$. We define the discrepancy of $\Phi_n^f$ as
\begin{equation*}
\Delta_f(n) = \max \left\lbrace \vert \textrm{card*} \lbrace \phi \in \Phi_n^f, \phi \in [\theta_1,\theta_2] \bmod{2\pi} \rbrace - (\theta_2 - \theta_1) r_f(n) \vert, 0 \leq \theta_1 < \theta_2 \leq 2\pi \right\rbrace
\end{equation*}
where the asterisk denotes that if the angle is $\alpha$ or $\beta$, it counts for $\frac{1}{2}$.
\end{definition}
\noindent We would like to use the same technique as in the first section, using the bound
\begin{equation*}
\Delta_f(n) \ll \frac{r_f(n)}{T} + \sum_{t \leq T} \frac{\vert Z_t^f(n) \vert}{t}
\end{equation*}
where $Z_t^f(n) = \sum_{\phi \in \Phi_n^f} \e^{i t \phi}$.
\\
\\
Unfortunately, the function $\frac{\vert Z_t^f(n) \vert}{r_f(n)}$ might not be multiplicative (and will generally not be). However, in the (very restrictive) case where there is one class per genus, what we did before allows us to give a bound on the discrepancy $\Delta_f(n)$, since it will be equal to $\Delta(n)$.

\subsection{One class per genus}

We recall that two forms of discriminant $-d$ are in the same genus if they represent the same values in $\mathbb{Z}/d\mathbb{Z}$. Therefore, two forms in the same class will always be in the same genus, and in fact each genus consists of the same number of classes of forms.
\\
\\
It is known that there is only a finite number of fundamental discriminants $-d$ with one class per genus. These numbers are
\begin{align*}
3,4,7,8,11,15,19,20,24,35,40,43, 51,52,67,84,88,91,115,120,&123,132,148,163,168,187,195, \\
228,232,235,267,280,312,340,372,403,408,420,427,435,483,520,&532,555,595,627,660,708,715,\\
760,795,840,1012,1092,1155,1320,1380,1428,1435,1540,&1848,1995,3003,3315,5460 \, .
\end{align*}
It is classical that, if there is only one class per genus, any integer coprime with $2d$ will be represented by a unique class of quadratic forms, or, equivalently, there is a unique class of ideals containing an ideal of norm $n$, and therefore a unique class of ideal numbers containing an integral ideal number of norm $n$.

\begin{theorem}
Suppose that $d$ is such that there is one class per genus. Let $\mathbb{G}_f$ be the set of $n \in \mathbb{G}$ with $r_f(n)>0$. Let $\varepsilon > 0$. Then, for all the integers in $\mathbb{G}_f$ less than $x$, with at most $o\left( \vert \mathbb{G}_f \cap [0,x] \vert \right)$ exceptions, we have
\begin{equation*}
\Delta_f(n) \leq \frac{r_f(n)}{(\log x)^{\frac{1}{2} \log \left( \frac{\pi}{2} \right) - \varepsilon}} \, .
\end{equation*} 
\end{theorem}

\begin{remark}
In \cite{Bern}, Bernays in fact proves a stronger result than the one given in Lemma \ref{Bern}:
\begin{equation*}
\vert \mathbb{G}_f \cap [0,x] \vert = \frac{\kappa_d}{h} \frac{x}{\sqrt{\log x}} + O \left( \frac{x}{\log^{3/4} x} \right) \, .
\end{equation*}
\end{remark}

\begin{proof}
We just need to prove that, if $r_f(n)>0$, we have $r_f(n)=r(n)$. Then, we will always have $\Delta_f(n)=\Delta(n)$ if $n \in \mathbb{G}_f$, which gives us the desired result. Let $n=mk$ with
\begin{equation*}
m=\prod_{\left( \frac{-d}{p}\right)=1} p^{\alpha_p} \quad \text{and} \quad k=\prod_{\left( \frac{-d}{q}\right)=0} q^{\beta_q} \prod_{\left( \frac{-d}{r}\right)=-1} r^{2\gamma_r'} \, .
\end{equation*}
If $\mathcal{C}_g$ denotes the class $\mathcal{C}_f \prod_{\left( \frac{-d}{q}\right)=0} \mathcal{C}_q^{-\beta_q}$, where $\mathcal{C}_q$ is the class of $\mathcal{Q}$, with $[q]=\mathcal{Q}^2$, then $r_f(n)= r_g(m)$.
\\
\\
If $n \in \mathbb{G}_f$, then $r_g(m)>0$. If $(m,2)=1$, since then $(m,2d)=1$ and there is one class per genus, $r_g(m)=r(m)=r(n)$. Hence, $r_f(n)= r(n)$.
\\
\\
If $2 \vert m$, then $\left( \frac{-d}{2}\right)=1$, i.e. $d \equiv \pm 1 \bmod{8}$. The only case with class number $h>1$ is $d=15$. In this case, $h=2$. The two reduced forms of discriminant $-15$ are
\begin{equation*}
f_0(x,y)=x^2 + xy + 4y^2 \quad \text{and} \quad f_1(x,y)=2x^2 + xy + 2y^2
\end{equation*}
and, since $h=2$, $\mathcal{C}_{f_i}^{-1}=\mathcal{C}_{f_i}$ for $i \in \lbrace 0,1 \rbrace$. If $m \in \mathbb{G}_{f_i}$ with $m=2^l s$, $s$ odd, an ideal $\mathcal{I}$ of norm $m$ in $\mathcal{C}_{f_i}$ is of the form
\begin{equation*}
\mathcal{I} = [2,\frac{1+i\sqrt{15}}{2}]^{\alpha_2}[2,\frac{1-i\sqrt{15}}{2}]^{l-\alpha_2} \mathcal{J}
\end{equation*}
where $\mathcal{J}$ is an ideal of norm $s$ in the class $\mathcal{C}_{f_i} \mathcal{C}_{f_1}^k$. Reciprocally, for any ideal $\mathcal{J}$ of norm $s$ in the class $\mathcal{C}_{f_i} \mathcal{C}_{f_1}^l$ and for every choice of $0 \leq \alpha_2 \leq l$, such an ideal is an ideal of norm $m$ in $\mathcal{C}_{f_i}$.
\\
\\
Therefore, if $m=2^l s$, $s$ odd, then
\begin{equation*}
r_{f_i}(2^l r) = (l+1) r(s) = r(m) \, .
\end{equation*}
\end{proof}

\bigskip

\bibliographystyle{alpha}
\nocite{*}
\bibliography{idealnumbers}

\textsc{\footnotesize D\'{e}partement de math\'{e}matiques et statistiques,
Universit\'{e} de Montr\'{e}al,
CP 6128 succ. Centre-Ville,
Montr\'{e}al QC H3C 3J7, Canada
}

\textit{\small Email address: }\texttt{\small dimitrid@dms.umontreal.ca}

\end{document}